%% file: mixedproducts.tex
\documentclass[11pt]{amsart}
\usepackage{amsfonts,amsthm,amssymb,mathrsfs,amscd,comment,amsmath, url}
\usepackage[utf8]{inputenc}
\usepackage[T1]{fontenc}

\usepackage{varioref}
\usepackage{hyperref}
\usepackage{cleveref,autonum}

\newtheorem{lemma}{Lemma}
\newtheorem{lemma*}{Lemma}
\newtheorem{theorem}{Theorem}
\newtheorem*{theorem*}{Theorem}
\newtheorem*{conj*}{Conjecture}

\newtheorem{remark*}{Remark}
\newtheorem{cor}{Corollary}
\newtheorem*{cor*}{Corollary}

\newcommand*{\bfrac}[2]{\genfrac{}{}{0pt}{}{#1}{#2}}
\global\long\def\veps{\varepsilon}

\global\long\def\N{\mathbb{N}}
\global\long\def\C{\mathbb{C}}
\global\long\def\P{\mathbb{P}}
\global\long\def\Ac{{\mathcal{A}}}
\global\long\def\Hc{\mathcal{H}}
\global\long\def\Lc{\mathcal{L}}
\global\long\def\Pc{\mathcal{P}}
\global\long\def\fa{f_{\Ac}}
\global\long\def\om{\omega}
\global\long\def\po{p \equiv 1(\mmod 4)}
\global\long\def\pt{p \equiv 3(\mmod 4)}
\global\long\def\simA{\overset{A}{\sim}}

\DeclareMathOperator{\mmod}{mod}
\DeclareMathOperator{\Res}{Re}

\date{}

\begin{document}
\title[Intersections of binary quadratic forms in primes]{Intersections of binary quadratic forms in primes and the paucity phenomenon}
\author{Alisa Sedunova}
\address{Saint Petersburg State University}
\email{alisa.sedunova@phystech.edu}

\begin{abstract}
\noindent The number of solutions to $a^2+b^2=c^2+d^2 \le x$ in integers is a well-known result, while if one restricts all the variables to primes Erd\H{o}s \cite{erdos} showed that only the diagonal solutions, namely, the ones with $\{a,b\}=\{c,d\}$ contribute to the main term, hence there is a paucity of the off-diagonal solutions. 
Daniel \cite{MR1833070} considered the case of $a,c$ being prime and proved that the main term has both the diagonal and the non-diagonal contributions.
Here we investigate the remaining cases, namely when only $c$ is a prime and when both $c,d$ are primes and, finally, when $b,c,d$ are primes by combining techniques of Daniel, Hooley and Plaksin.
\end{abstract}

\maketitle

\setcounter{tocdepth}{1}
\tableofcontents

\input{intro}
\input{notations}
\input{lemmas}

\input{r0r1}

\input{r0r2}

\input{r1r2}

\bibliographystyle{abbrv}

\bibliography{bibliography}{}
\end{document}

%% file: intro.tex
\section{Introduction}
\noindent The paucity of positive solutions to Diophantine equation
$a^k+b^k=c^k+d^k$ outside of the diagonal given by $\{a,b\} = \{c,d\}$
was predicted for $k \ge 3$ and demonstrated by Hooley \cite{MR623719}. 
For $k=2$ the situation is very different. 
Indeed, the diagonal solutions to $a^2+b^2=c^2+d^2$ contribute only to the error term, while the off-diagonal ones contain the main term. 
When $a,b,c,d$ are all primes the situation changes drastically. 
Erd\H{o}s \cite{erdos} established that the quadric $a^2+b^2=c^2+d^2$ has almost all its prime solutions on the diagonal. 
The same phenomenon was later discovered for the prime solutions to the intersection of two quadrics by  Blomer and Br\"{u}dern \cite{MR2418801}, who proved that the main term is coming only from the diagonal solutions to $a^2+b^2=c^2+d^2=e^2+f^2$. 
Here we investigate positive solutions to $a^2+b^2=c^2+d^2$ when \emph{some} of $a,b,c,d$ are primes, namely, when only $c$ is a prime, and when both $c,d$ are primes and, finally, $b,c,d$ are primes.
We confirm that there is no paucity of the non-diagonal solutions in the first two cases. 

\noindent Let $\P$ denote the set of all primes and the letters $p,q,r$ always stand for elements of $\P$.
\noindent For $n \in \N$ define the following functions
\begin{align}
	\label{r_0def} r_0(n) &= \#\{ (a,b) \in \N^2 \colon n = a^2+b^2\}, \\
	\label{r_1def} r_1(n) &= \#\{ (a,p) \in \N \times \P \colon n=a^2+p^2\}, \\
	\label{r_2def} r_2(n) &= \#\{ (p,q) \in \P^2 \colon n=p^2+q^2\}.
\end{align}
\noindent For $i,j \in \{0,1,2\}$ consider the mean values

\begin{equation}
S_{i,j}(x) = \sum_{n \le x} r_i(n)r_j(n).
\end{equation} 

\noindent The asymptotic formulas for $S_{i,j}(x)$ can be interpreted in terms of the number of solutions to
$a^2+b^2=c^2+d^2$,
where $a,b,c,d \in \N$ and some of them are restricted to primes, depending on $i,j$.

\noindent In the present work we consider $S_{i,j}(x)$ with $i \neq j$; $i, j \in \{ 0, 1, 2\}$.

\begin{theorem} \label{r_0r_1}
As $x \to \infty$ we have that
\begin{equation}
	S_{0,1}(x) = \sum_{n \le x} r_0(n)r_1(n) = \frac{x}{2}+O\biggl(\frac{x \log \log x}{\log x}\biggr).
\end{equation}
\end{theorem}

\noindent As a corollary of the theorem above we get the following.
\begin{cor} \label{dispr0r1}
For any $c>0$ the dispersion satisfies
\begin{equation}
	\sum_{n \le x} \biggl( r_1(n) - \frac{cr_0(n)}{\log n }\biggr)^2 \asymp \frac{x}{\log x}.
\end{equation}
\end{cor}

\noindent The proof of the Corollary is a byproduct of \eqref{s11}, Theorem \ref{r_0r_1}, equation \eqref{s00} and partial summation. 
The corollary implies that $cr_0(n)/\log n$ is not a good enough approximation for considerably more complicated function $r_1(n)$.

\begin{theorem} \label{r_0r_2}
As $x \to \infty$ we have
\begin{equation}
	S_{0,2}(x) = \sum_{n \le x} r_0(n)r_2(n) = \frac{12G}{\pi^2} \frac{x}{\log x} + O\biggl(\frac{x \log \log x}{\log^2 x}\biggr),
\end{equation}
where $G$ is the Catalan constant given by
\begin{equation}
	G=L(2, \chi_4) = \sum_{k=0}^\infty \frac{(-1)^k}{(2k+1)^2} = 0.915 \ldots
\end{equation}
\end{theorem}

\begin{theorem} \label{r_1r_2}
As $x \to \infty$ we have that
\begin{equation}
	\frac{2 \pi x}{\log^2 x} \ll S_{1,2}(x) = \sum_{n \le x} r_1(n)r_2(n) \ll \frac{x}{\log^2 x} (\log \log x)^2.
\end{equation}
\end{theorem}

Notice that the circle method struggles with mean values as in theorems \ref{r_0r_1}, \ref{r_0r_2}, \ref{r_1r_2} mainly because of the lack of variables and their primality causing major arcs to be too small.  
Instead, we develop other techniques for these mean values.
We take an approach of Hooley \cite{MR0088512}, who used it to prove the Hardy-Littlewood problem of representing $n=p+a^2+b^2$ conditionally on GRH.
This requires careful preparation for an application of the Barban-Davenport-Halberstam theorem (see \eqref{BDH}).

One could try to consider another multiplicative function in place of $r_0(n)$. 
In that direction, for instance, very recently Li \cite{MR4054578} proved an asymptotic formula for the sum of $r_2(n) \tau(n+1)$.
The outline of the present proof is quite general and can be adapted to some other multiplicative functions instead of $r_0$ as well as some of the shifts provided that we have the corresponding Bombieri-Vinogradov type theorem.



Let us survey the well-known results that we are going to use throughout the paper. 
It is well known that 
\begin{equation}
	\sum_{n \le x}r_0(n) = \frac{\pi x}{4}-\sqrt{x}+O(x^\alpha),
\end{equation}
where $\alpha<1/3$ (see \cite{MR1199067} for improvements on $\alpha$) and
\begin{equation} \label{s00}
	S_{0,0} (x) =  \sum_{n \le x} r_0^2(n) = \frac{x}{4} (\log x+C)+O(x^{1/2+\veps}),
\end{equation}
where $C$ is an explicit constant
(this can be done by computing the corresponding Dirichlet series and then applying standard complex analytic techniques).

Further, Daniel \cite{MR1833070} considered $S_{1,1}(x)$ and proved that
\begin{equation} \label{s11}
	S_{1,1}(x) = \sum_{n \le x} r_1^2(n) = \biggl(\frac{\pi}{2}+\frac{9}{4}\biggr)\frac{x}{\log x}+O\biggl( \frac{x (\log \log x)^2}{\log^2 x}\biggr),
\end{equation}
while by the prime number theorem we have
\begin{equation}  \label{s1}
	\sum_{n \le x}r_1(n) = \frac{\pi}{2} \frac{x}{\log x}+O\biggl(\frac{x}{\log^2 x}\biggr).
\end{equation}
Notice that \eqref{s11} improves on Rieger \cite{MR0229603}, who essentially proved $S_{1,1}(x) \ll x/ \log x$ via sieve techniques.
Here the part with $\pi/2$ in the main term comes from the diagonal solutions to $a^2+p^2=b^2+q^2$, namely, the ones with $a=b$, $p=q$, while $9/4$ is coming from the off-diagonal ones and is the main object to study in \cite{MR1833070}. It is also showed in \cite[Theorem 2]{MR1833070} that the number of $n \le x$ with $r_1(n)>0$ is asymptotic to $\pi x/(2\log x)$, which, together with \eqref{s1}, implies that the $9/4$ in \eqref{s11} is coming from a very thin set of extremely large values of $r_1$.

Erd\H{o}s \cite{erdos} found that 
\begin{equation} \label{s22}
	S_{2,2} (x) = \frac{2 \pi x}{\log^2 x}+O\biggl( \frac{x \log \log x}{\log^3 x}\biggr),
\end{equation}
while it was shown by Plaksin \cite[Lemma 11]{MR616225} that
\begin{equation} \label{s2}
	\sum_{n \le x} r_2(n) = \frac{ \pi x}{\log^2 x}+O\biggl(\frac{x \log \log x}{\log^3 x}\biggr)
\end{equation}
(see also \cite[Lemma 3.3]{MR4054578}).
Actually Rieger \cite{MR0229603} showed that
\begin{equation} \label{r2^2viar2}
	S_{2,2} (x) = \sum_{n \le x} r_2^2(n)= 2 \sum_{n \le x} r_2(n)+ O \biggl(\frac{x}{\log^3 x} \biggr)
\end{equation}

More recently, Blomer and Br\"{u}dern  \cite{MR2418801} obtained an asymptotic for the third moment, namely, they proved
\begin{equation} \label{r_2^3}
	\sum_{n \le x} r_2^3(n) = \frac{4 \pi x}{\log^2 x}+O\biggl( \frac{x(\log \log x)^6}{\log^3 x}\biggr),
\end{equation}
hence the main term is coming only from the diagonal solutions. They estimated the contribution of the off-diagonal solutions to $p_1^2+p_2^2=p_3^2+p_4^2=p_5^2+p_6^2 \le x$ using suitable representations of $p_i, i=1,\ldots,6$ over the Gaussian integers (where the factorization of that equation in $\mathbb{Z}[i]$ is being used) and then applied the upper bound sieve.


On noticing that $r_2(n)=1$ only when for $\sqrt{x}$ values up to $x$ and combining \eqref{s2} with \eqref{s22}, we obtain
\begin{equation} \label{r2>0}
	\sum_{\substack{n \le x \\ r_2(n)>0}} 1 = \frac{\pi}{2} \frac{x}{\log^2 x}+O \biggl( \frac{x \log \log x}{\log^3 x}\biggr).
\end{equation}
However, the same trick fails when dealing with $r_1$ due to the non-trivial contribution of the off-diagonal solutions.

First, we point out some rather simple bounds for $S_{i,j}(x)$.
Applying trivially $0 \le r_2(n) \le r_1(n) \le r_0(n)$ and using Cauchy-Schwarz, we find that

\begin{equation}
	\begin{split}
		{x}/{\log^2 x} \ll S_{2,2} (x) \le &S_{0,2}(x) \le S_{0,0}(x)^{1/2} S_{2,2}(x)^{1/2} \ll {x}/{\sqrt{\log x}},\\
		{x}/{\log x} \ll S_{1,1} (x) \le &S_{0,1}(x) \le S_{0,0}(x)^{1/2} S_{1,1}(x)^{1/2} \ll x,\\
		{x}/{\log^2 x} \ll S_{2,2} (x) \le &S_{1,2}(x) \le S_{1,1}(x)^{1/2} S_{2,2}(x)^{1/2} \ll {x}/{\log^{3/2}x},
	\end{split}
\end{equation} 
where we  used \eqref{s00}, \eqref{s11} and \eqref{s22}. 

Theorem \ref{r_0r_2} implies that there is no paucity of the off-diagonal solutions, so they contribute to the main term, while the diagonal ones are $\ll x/\log^2 x$ and lie in the error term. 

Similarly, in Theorem \ref{r_0r_1} the main term comes from the non-diagonal contribution, while the diagonal one, namely, $S_{1,1}(x)$ goes to the error term.

In the case of $S_{1,2}(x)$ we do not have any information regarding the paucity of the off-diagonal solutions, getting an asymptotic formula for $r_1(n)r_2(n)$ seems to be quite hard since it involves non-trivial estimates for $r_1^2(n)$ over the support of $r_2(n)$. 


%% file: notations.tex
\section{Notations}
\noindent Let $\chi_4$ be the non-principal Dirichlet character modulo $4$. 
Recall that by \eqref{r_0def} we have $r_0(n) = \sum_{d|n}\chi_4(d)$, hence $r_0(n)$ is multiplicative.
In what follows $1_X$ stands for the indicator function of $X$, $\omega(n)$ is the number of distinct prime factors of $n$, $\varphi(n)$ is the Euler totient function, $\tau(n)$ is the number of divisors of $n$.
We recall the well-known fact that $n$ can be represented as a sum of two squares if and only if all prime divisors of $n$ that are equal to $3 (\mmod 4)$ appear in the prime factorization of $n$ with even powers.

\noindent Denote by $K$ the Landau-Ramanujan constant given by the Euler product
\begin{equation} \label{landaucst}
	K=\frac{\pi}{4} \prod_{p \equiv 1(\mmod 4)} \Big(1-\frac{1}{p^2}\Big)^{1/2} = \frac{1}{\sqrt{2}} \prod_{p \equiv 3(\mmod 4)} \Big(1-\frac{1}{p^2}\Big)^{-1/2} = 0.764223653...
\end{equation}

\noindent Define an indicator function of the representation as a sum of two squares as follows
\begin{equation} \label{bdef}
	b(n)=\begin{cases}
		1, & n=a^2+b^2 \text{ for some }a,b \in \N,\\
		0, &\text{otherwise.}
	\end{cases}
\end{equation}

\noindent Then, the classical result of Landau \cite{landau} states
\begin{equation} \label{landau}
	\sum_{n \leq x} b(n) = \frac{Kx}{\sqrt{\log x}}\biggl(1+O \Big( \frac{1}{\log x}\Big) \biggr).
\end{equation}

\noindent Denote by $\Ac$ the set of natural numbers $n \in \N$ with all prime divisors of $n$ satisfy $p \equiv 1(\mmod 4)$
\begin{equation}
	\Ac=\{n \in \N: \quad p|n \to p \equiv 1(\mmod 4)\}.
\end{equation}

\noindent Clearly, $1 \in \Ac$. Let $f_{\Ac}(n)$ be the characteristic function of $\Ac$
\begin{equation}
	f_{\Ac}(n) = \begin{cases}
		1, &n\in \Ac,\\
		0, &n \notin \Ac.
	\end{cases}
\end{equation}
\noindent Notice that $f_{\Ac}(n)$ is a totally multiplicative function with values in $\{0,1\}$, $f_{\Ac}(p^k)=1$ for $p \equiv 1(\mmod 4)$ and is zero otherwise.
Notice also that $\Ac$ has the same density as the set of numbers that can be represented as a sum of two squares.
Indeed, by the standard application of the convolution method it can be shown that as $x \to \infty$ we have
\begin{equation}
	\Ac(x)= \#\{n \le x, \, n\in \Ac \} = \frac{1}{4K} \frac{x}{\sqrt{\log x}}+O\biggl( \frac{x}{(\log x)^{3/2}}\biggr).
\end{equation}

%% file: lemmas.tex
\section{Auxiliary lemmas}
\noindent We start with some technical estimates related to our initial set $\Ac$.
\begin{lemma} \label{2^om/n_onA}
For $x \to \infty$, we have
\begin{equation}
\sum_{n \leq x}\frac{2^{\omega(n)}}{n} f_{\Ac}(n) = \frac{1}{\pi} \log x +O(1).
\end{equation}
\end{lemma}
The analogous estimate without $f_\Ac(n)$ is quite standard, namely, we have
\begin{equation} \label{2^om/n}
\sum_{n \leq x}\frac{2^{\omega(n)}}{n}  = \frac{3}{\pi^2} \log^2 x +O(\log x).
\end{equation}
We expect to get some power of $\log$ saving if we sum over $n \in \Ac$ instead of integers. 
This is one of the two main steps that allow us to obtain the main term of the correct order of magnitude in Theorem \ref{r_0r_2}.
Unfortunately, it appears to be difficult to modify the proof of \eqref{2^om/n} to the one with $f_\Ac$ and we proceed by the means of the celebrated Levin-Fainleib theorem \cite{MR0229600}.

\begin{proof}
Define a multiplicative function $f(n) = 2^{\om(n)} f_{\Ac}(n) /n$ and write its Dirichlet series as
\begin{equation}
F(s) = \sum_{n=1}^\infty \frac{2^{\omega(n)} f_{\Ac}(n)}{n^s} = \prod_{\po} \biggl(1+\frac{2}{p^s}+\frac{2}{p^{2s}}+\ldots \biggr) = \prod_{\po} \biggl( \frac{p^s+1}{p^s-1} \biggr).
\end{equation}
On defining an $L$--function attached to $\chi_4$ by
\begin{equation}
L(s,\chi_4) = \prod_{\po}\biggl(1-\frac{1}{p^s}\biggr)^{-1} \prod_{\pt}\biggl(1+\frac{1}{p^s}\biggr)^{-1}
\end{equation}
we can rewrite $F(s)$ as
\begin{equation} \label{genF}
F(s) =  \prod_{\po} \biggl( \frac{p^s+1}{p^s-1} \biggr) = \frac{L(s,\chi_4) \zeta(s)}{\zeta(2s)} \biggl(1+\frac{1}{2^s}\biggr)^{-1}.
\end{equation}
Further, for $Q >w \geq 1$ we have
\begin{equation}
\sum_{\bfrac{p \geq 2}{\bfrac{\nu \geq 1}{w<p^\nu \leq Q}}} f(p^\nu) \log p^\nu = 2 \sum_{\bfrac{\po}{w<p \leq Q}} \frac{\log p}{p} +2\nu \sum_{\bfrac{p \geq 2, \po}{\bfrac{\nu \geq 2}{w<p^\nu \leq Q}}} \frac{\log p}{p^\nu} = \log \frac{Q}{w}+O(1)
\end{equation}
and
\begin{equation}
\sum_{p \geq 2} \sum_{\nu, k \geq 1} f(p^k)f(p^\nu) \log p^\nu = \sum_{\bfrac{p \geq 2}{\po}} \sum_{\nu,k \geq 1} \frac{4 \log p^\nu}{p^{k+\nu}} \ll \sum_{p} \frac{\log p}{p^2} < \infty.
\end{equation}
Hence one can apply \cite[Theorem 21.1]{MR2493924} with $\kappa=1$ and obtain
\begin{equation}
\sum_{n \leq x} f(n) = C \log x+O(1), \quad C=\frac{1}{\Gamma(2)} \prod_p \biggl( \Big(1-\frac{1}{p}\Big) \sum_{\nu \geq 0} f(p^\nu) \biggr).
\end{equation}
Comparing with \eqref{genF} we have
\begin{equation}
C= \frac{L(1,\chi_4)}{\zeta(2)} \biggl(1+\frac{1}{2}\biggr)^{-1} = \frac{1}{\pi},
\end{equation}
where we used that $L(1,\chi_4) = \pi/4$, $\zeta(2)=\pi^2/6$.
\end{proof}

\begin{lemma} \label{2^om/phi(n)_onA}
For $x \to \infty$, we have
\begin{equation}
\sum_{n \leq x}\frac{2^{\omega(n)}}{\varphi(n)} f_{\Ac}(n) = \frac{12 G}{\pi^3} \log x +O(1),
\end{equation}
where $G$ is the Catalan constant given by
\begin{equation}
G=L(2, \chi_4) = \sum_{k=0}^\infty \frac{(-1)^k}{(2k+1)^2} = 0.915 \ldots
\end{equation}
\end{lemma}

\begin{proof}
We proceed by the convolution method and write $f_0(n)=2^{\om(n)} \fa(n) /n$, $f(n)=2^{\om(n)} \fa(n) /\varphi(n)$ and $f=f_0 \ast g$. Computing $g$ on prime powers we get that $g(p^0)=1$ and
\begin{equation} \label{gprime}
g(p^k) = (-1)^{k+1} \frac{2 \fa(p)}{p^k (p-1)}, \quad k \ge 1.
\end{equation}
Further, we write
\begin{equation}
\begin{split}
G_0 &= \prod_p \Big(1+g(p)+g(p^2)+\ldots\Big) \\
&= \prod_{\po} \biggl(1+\frac{2}{p(p-1)}\Big(1-\frac{1}{p}+\frac{1}{p^2}-\ldots \Big)\biggr) \\
&= \prod_{\po} \biggl(1+\frac{2}{p^2-1}\biggr) = \prod_{\po} \biggl(\frac{p^2+1}{p^2-1}\biggr).
\end{split}
\end{equation}
Using \eqref{genF} with $s=2$ we get
\begin{equation}
G_0  = F(2) = \frac{L(2,\chi_4) \zeta(2)}{\zeta(4)} \biggl(1+\frac{1}{2^2}\biggr)^{-1} = \frac{12G}{\pi^2},
\end{equation}
where in the last equality we used that $\zeta(2) = \pi^2/6$, $\zeta(4) = \pi^4 /90$ and $L(2,\chi_4)=G$.
Finally, on applying, for example, \cite[Theorem 1.1]{MR2061214} one immediately sees that
\begin{equation}
	\sum_{d \le x} g(d) = G_0 + O \biggl(\frac{1}{\log x}\biggr)
\end{equation}
and thus
\begin{equation}
\begin{split}
\sum_{n \le x} f(n) &= \sum_{n \le x} f_0 \ast g (n) = \sum_{d \le x} g(d) \sum_{n \le x/d}f_0(n) \\
&= \frac{\log x}{\pi} \sum_{d \le x} g(d)+ O\Big( \sum_{d \le x} |g(d)| \log d\Big) =  \frac{G_0 }{\pi} \log x+O(1),
\end{split}
\end{equation}
where in the last equality we used
\begin{equation}
	\sum_{d \le x} |g(d)|\log d = \sum_{d \le x} \frac{\fa(d) \log d}{d} \prod_{p \mid d}\frac{2}{p-1} \le \sum_{d \le x} \frac{\log d}{d 2^{\omega(d)}} \ll 1.
\end{equation}
\end{proof}

\begin{lemma}\label{uv-sols}
Let $\rho_1(d)$ be the number of solutions to
\begin{equation}
u^2+v^2 \equiv 0 (\mmod d), \quad u,v \mmod d, \quad (v,d)=1.
\end{equation}

Then $\rho(d)$ is a multiplicative functions defined on prime powers $d=p^\alpha$, $\alpha \ge 1$ as
\begin{equation} \label{solcond}
	\rho(p^\alpha)= ( 1+ \chi_4(p) ) \varphi(p^\alpha) = 2 \varphi(p^\alpha) \fa(p), \quad  p \neq 2, \quad i=1,2
\end{equation}
and $\rho_1(2)=1$, $\rho_1(2^\alpha)=0$ for $\alpha \ge 2$.


\end{lemma}

\begin{proof}
This is a standard result, which we present for the sake of completeness and 
follow closely Plaksin \cite[Lemma 14 and Lemma 15]{MR616225}. 
Let $d=p^\alpha$, $p \neq 2$ and $(u_\alpha,v_\alpha)$ be a solution to \eqref{solcond}. 
Consider the solutions $(u_{\alpha+1}, v_{\alpha+1})$ that contribute to $\rho(p^{\alpha+1})$.
We have $u_{\alpha+1} = u_\alpha+t_1p^\alpha$, $v_{\alpha+1} = v_\alpha+t_2p^\alpha$ with $0 \le t_1,t_2<p$. 
Hence $2u_\alpha t_1+2 v_\alpha t_2 \equiv 0 (\mmod p)$.
Since $p \neq 2$ this is a linear congruence in $t_1, t_2$ that has exactly $p$ solutions. 
Thus $\rho(p^{\alpha+1}) = p\cdot \rho(p^\alpha)=p^\alpha \rho(p)$. 
For $\rho(p)$ we have
\begin{equation}
\begin{split}
\rho(p) &= \# \{ (u/v)^2 \equiv -1 (\mmod p), \quad u,v \mmod p, \quad (v,p)=1\} \\
&=\# \{ x^2 \equiv -1 (\mmod p)\} \cdot (p-1) = 2(p-1)\fa(p), 
\end{split}
\end{equation}
concluding the statement for $p \neq 2$. Now let $p=2$. 
Calculating the first three values we get $\rho(2)=1$, $\rho(4)=\rho(8)=0$. 
Using \cite[Chapter VI, Section 4]{MR0190056} we obtain $\rho(2^\alpha)=0$ for $\alpha \ge 2$, hence finishing the proof.
\end{proof}

%% file: r0r1.tex
\section{Number of solutions to $a^2+p^2=c^2+d^2 \le x$} \label{r_0r_1sec}
We follow closely the approach of Hooley \cite{MR3763074} as well as use the result on the second moment of twisted Hooley delta function from \cite{MR3060456}. 

We emphasize that the asymptotic formula for $S_{1,1}(x)$ of Daniel \cite{MR1833070} makes a great deal out of the diagonal solutions, while here we can not afford exactly the same procedure because in our case the diagonal solutions contribute only to the error term.

The technique developed in this section is going to be applied in slightly modified form for $S_{0,2}(x)$ as well, so we present it in detail here to skip similar computations in the next section.

We wish to subdivide the range of summation over $d$ into three parts: "small" divisors, "large" divisors and the "middle" ones, say, around square root. 
Then due to the symmetry of the form in question we can switch to the co-divisors and observe that the main term should come from the "small" and "large" ranges, which are, indeed, equal. 
We next apply the Barban-Davenport-Halberstam theorem in the main term.

Let $A$ be some large positive integer. 
By the definition of $r_0$ and $r_2$ given in \eqref{r_0def} and \eqref{r_2def} respectively we have
\begin{equation}
	S_{0,1}(x) =\sum_{n \leq x} r_0(n) r_1(n)= \sum_{n \leq x} r_1(n) \sum_{d \mid n} \chi_4(d).
\end{equation}

\noindent Notice that since we sum over those $n$ that can be written as $a^2+p^2$, then for $2 \nmid n$ we have $\chi_4(n)=1$ and if $\delta \mid n$, then $\chi_4(\delta)=\chi_4(n/\delta)$. 
Thus, we split the divisors into three ranges
\begin{equation}
	S_{0,1}(x) = \sum_{\bfrac{n \leq x}{2 \nmid n}} r_1(n) \sum_{i=1}^3 \sigma_i(n) + \sum_{\bfrac{n \leq x}{2 \mid n}} r_1 (n) \sum_{i=1}^3 \sigma_i(n),
\end{equation}
where
\begin{equation} \label{defsigma1r0r2}
	\sigma_1(n)=\sum_{\bfrac{d \mid n}{d \leq \sqrt{n}/\log^A x}}\chi_4(d), \quad \quad \sigma_3(n)=\sum_{\bfrac{d \mid n}{d \geq \sqrt{n}\log^A x}}\chi_4(d),
\end{equation}
and, finally,
\begin{equation} \label{defsgima2r0r2}
	\sigma_2(n)=\sum_{\bfrac{d \mid n}{d \simA \sqrt{n}}}\chi_4(d),
\end{equation}
where $d \simA \sqrt{n}$ means that $\sqrt{n}/\log^A x < d \le \sqrt{n} \log^A x$.

\noindent When $2 \nmid n$ we can switch to the co-divisors in  $\sigma_3(p,q)$ and obtain
\begin{equation}
	\sigma_3(n)=\sum_{\bfrac{ \delta d = n}{d \geq \sqrt{n}\log^A x}}\chi_4(d) = \sum_{\bfrac{ \delta \mid n}{\delta \leq \sqrt{n}/\log^A x}}\chi_4(n/\delta) = \sum_{\bfrac{ \delta \mid n}{\delta \leq \sqrt{n}/\log^A x}}\chi_4(\delta) = \sigma_1(n).
\end{equation}

\noindent Thus
\begin{equation} \label{s01inr0r2}
	\begin{split}
		S_{0,1}(x) &= 2 \sum_{\bfrac{n \le x}{2 \nmid n}}r_1(n) \sigma_1(n)+ \sum_{\bfrac{n \le x}{2 \nmid n}} r_1(n) \sigma_2(n)+ \sum_{\bfrac{n \leq x}{2 \mid n}} r_1(n) \sum_{i=1}^3 \sigma_i(n) \\
				&= 2 \sum_{\bfrac{n \le x}{2 \nmid n}} r_1(n) \sigma_1(n)+ \sum_{n \le x} r_1(n) \sigma_2(n) +  \sum_{\bfrac{n \leq x}{2 \mid n}} r_1(n) \bigl(\sigma_1(n)+\sigma_3(n)\bigr)\\
				&= 2\sum_{n \le x} r_1(n) \sigma_1(n)+\sum_{n \le x} r_1(n) \sigma_2(n) + \sum_{\bfrac{n\leq x}{2 \mid n}} r_1(n) \bigl(\sigma_3(n)-\sigma_1(n)\bigr)\\
				&= 2\sum_{n \le x} r_1(n) \sigma_1(n)+R(x),
	\end{split}
\end{equation}
where
\begin{equation}
	\begin{split}
		R(x) &= \sum_{n \le x} r_1(n) \sigma_2(n) -  \sum_{\bfrac{n \leq x}{2 \mid n}} r_1(n) \sum_{\bfrac{d \mid n}{2^{-k} \sqrt{n}/\log^A x < d < \sqrt{n}/\log^A x}} \chi_4(d)\\
			&= R_1(x) + R_1'(x),
	\end{split}
\end{equation}
where $k = \max \{ \ell \colon 2^\ell \mid n \}$ (note that in our case $k \le 1$).
In the above we used the fact that
\begin{equation}
	\begin{split}
	\sum_{\bfrac{n \leq x}{2 \mid n}} r_1(n) (\sigma_3(n)-\sigma_1(n)) &= - \sum_{\bfrac{n \leq x}{2 \mid n}} r_1(n) \sum_{\bfrac{d \mid n}{2^{-k} \sqrt{n}/\log^A x < d < \sqrt{n}/\log^A x}} \chi_4(d)
	\end{split}
\end{equation}
since
\begin{equation}
	\begin{split}
		\sum_{\bfrac{d \mid n}{d \ge \sqrt{n}\log^A x}} \chi_4(d) &= \sum_{\bfrac{d \mid n}{\bfrac{d \ge \sqrt{n}\log^A x}{2 \nmid d}}} \chi_4(d) =  \sum_{\bfrac{\delta \mid n}{\bfrac{\delta \le \sqrt{n}/\log^A x}{2 \mid \delta, \,  2 \nmid (n/\delta)}}} \chi_4(n/\delta)\\
		&= \sum_{\bfrac{\tilde{\delta} \mid \tilde{n}}{\bfrac{\tilde{\delta} \le 2^{-k} \sqrt{n}/\log^A x}{2 \nmid \tilde{n}, \, 2 \nmid \tilde{\delta}}}} \chi_4(\tilde{n}/\tilde{\delta}) = \sum_{\bfrac{\delta \mid n}{\delta \le 2^{-k} \sqrt{n}/{\log^A x}}} \chi_4(\delta).
	\end{split}
\end{equation}

\noindent Let $D=\sqrt{x}/\log^A x$ and write $S_{0,1}(x) = 2M_1(x)+R(x)$, where the expected main term is 
\begin{equation} \label{M2def}
	M_1(x) = \sum_{n \le x} r_1(n) \sigma_1(n) = \sum_{d \le D} \chi_4(d) \sum_{\bfrac{n \le x}{n \equiv 0(\mmod d)}} r_1(n), 
\end{equation}
and the expected error term is $R(x) = R_1(x)+R_1'(x)$.

First notice that if we do not aim for an asymptotic formula, then the correct order of magnitude is mush easier to get (in fact, we can use only Brun-Titchmarsh theorem and Lemma \ref{uv-sols} to get the result). 
Splitting the quadratic congruence into two linear ones we have
\begin{equation}
		\sum_{\bfrac{n \le x}{n \equiv 0(\mmod d)}} r_1(n)  = \sum_{\bfrac{u,v (\mmod d)}{u^2+v^2 \equiv 0 (\mmod d)}} \sum_{\bfrac{a^2+p^2 \le x}{\bfrac{ a \equiv u (\mmod d)}{p \equiv v(\mmod d)}}}1.
\end{equation}

\noindent Further, when $d \le D$ we clearly have
\begin{equation} \label{splitM1}
		\sum_{\bfrac{n \le x}{n \equiv 0(\mmod d)}} r_1(n)  = \sum_{\bfrac{u,v (\mmod d)}{\bfrac{(v,d)=1}{u^2+v^2 \equiv 0 (\mmod d)}}} \sum_{\bfrac{a^2+p^2 \le x}{\bfrac{ a \equiv u (\mmod d)}{p \equiv v(\mmod d)}}}1 + O(\omega(d) \sqrt{x})
\end{equation}
and
\begin{equation} \label{M1bound}
		M_1(x) \ll \frac{x}{\log x} \sum_{d \le D} \frac{1}{d\varphi(d)} \sum_{\bfrac{u,v (\mmod d)}{\bfrac{(v,d)=1}{u^2+v^2 \equiv 0 (\mmod d)}}} 1  +  \sqrt{x} \sum_{d \le D} \omega(d) \ll x,
\end{equation}
where we applied Lemma \ref{uv-sols} , Lemma \ref{2^om/n_onA} and Brun-Titchmarsh theorem
\begin{equation} \label{quadcongr}
	\begin{split}
		\sum_{\bfrac{a^2+p^2 \le x}{\bfrac{ a \equiv u (\mmod d)}{p \equiv v(\mmod d)}}}1 &\le \sum_{\bfrac{a \le \sqrt{x}}{a \equiv u (\mmod d)}} \sum_{\bfrac{p \le \sqrt{x}}{ p \equiv v (\mmod d)}}1\\
		& \ll \frac{\sqrt{x}}{d}\frac{\sqrt{x}}{\varphi(d) \log (x/D)} \ll  \frac{1}{d\varphi(d)} \frac{x}{\log x}.
	\end{split}
\end{equation}

Further, for $R_1(x)$ first notice that
\begin{equation}
	\begin{split}
		\sum_{n \le x/ \log^A x} r_1(n) \sum_{\bfrac{d \mid n}{d \simA \sqrt{n}}} 1 &\le \sum_{d \le \sqrt{x} \log^{A/2}x} \sum_{\bfrac{n \le x/\log^A x}{n \equiv 0 (\mmod d)}}r_1(n) \\
		&\ll \sum_{d \le \sqrt{x} \log^{A/2}x} \frac{2^{\om(d)}}{d} \fa(d) \frac{x}{\log^{A+1}x} \ll \frac{x}{\log^A x},
	\end{split}
\end{equation}
where we used Lemma \ref{2^om/n_onA}, Lemma \ref{uv-sols} and \eqref{quadcongr}.
Acting similarly in the remaining range one obtains
\begin{equation}
	\begin{split}
		\sum_{x/ \log^A x < n \le x} r_1(n) \sum_{\bfrac{d \mid n}{d \simA \sqrt{n}}} 1 &\le \sum_{d \le \sqrt{x}\log^A x} \sum_{\bfrac{x/\log^A x < n \le x}{\bfrac{d^2/\log^A x \le n \le d^2 \log^A x}{n \equiv 0 (\mmod d)}}}r_1(n) \\
		&= \sum_{\bfrac{d \le \sqrt{x}\log^A x}{d \ge \sqrt{x}/\log^A x}} \sum_{\bfrac{x/\log^A x < n \le x}{\bfrac{d^2/\log^A x \le n \le d^2 \log^A x}{n \equiv 0 (\mmod d)}}}r_1(n) \\
		&\ll \sum_{\bfrac{d \le \sqrt{x}\log^A x}{d \ge \sqrt{x}/\log^A x}} \frac{2^{\om(d)}}{d} \fa(d) \frac{x}{\log x} \ll \frac{x}{\log x} \log \log x.
	\end{split}
\end{equation}
Hence $R_1(x) = O((x / \log x) \log \log x)$ and analogous estimate holds for $R_1'(x)$.
We conclude
\begin{equation}
	S_{0,1}(x) = 2M_1(x) + O\biggl( \frac{x}{\log x} \log \log x\biggr).
\end{equation}

Now let us focus on $M_1(x)$ that gives us the main term. 
 We split the quadratic congruence into two linear ones as in \eqref{splitM1} and notice
 \begin{equation}
 	\sum_{d \le D} \omega(d) \sqrt{x} \ll \sqrt{x} D \log D =\frac{x}{\log^{A-1}x},
 \end{equation}
  hence
 \begin{equation}
 	M_1(x) = 2 \sum_{d \le D} \chi_4(d) \sum_{\bfrac{u,v (\mmod d)}{\bfrac{(v,d)=1}{u^2+v^2 \equiv 0 (\mmod d)}}} \sum_{\bfrac{a^2+p^2 \le x}{\bfrac{ a \equiv u (\mmod d)}{p \equiv v(\mmod d)}}}1 + O\biggl( \frac{x}{\log^{A-1}x}\biggr).
 \end{equation}
 
\noindent Further, for the internal sum one has
\begin{equation}
	\begin{split}
		\sigma_{u,v,d}(x) &=   	\sum_{\bfrac{a^2+p^2 \le x}{\bfrac{ a \equiv u (\mmod d)}{p \equiv v(\mmod d)}}}1 = \sum_{\bfrac{a \le \sqrt{x}}{a \equiv u (\mmod d)}} \sum_{\bfrac{p \le \sqrt{x-a^2}}{p \equiv v(\mmod d)}} 1 \\
			&=  \sum_{\bfrac{a \le \sqrt{x}}{a \equiv u (\mmod d)}} \biggl( \frac{1}{\varphi(d)} \sum_{p \le \sqrt{x-a^2}} 1 +E(\sqrt{x-a^2};d,v)\biggr),  
	\end{split}
\end{equation}
where
\begin{equation}
	E(t;d,v) =  \sum_{\bfrac{p \le t}{p \equiv v (\mmod d)}} 1 - \frac{1}{\varphi(d)} \sum_{p \le t} 1
\end{equation}
is the error term in the prime number theorem in arithmetic progressions.

\noindent By the Barban-Davenport-Halberstam theorem we know that for any fixed $B>0$ such that $t/\log^B t \le T \le t$ we have
\begin{equation} \label{BDH}
	\sum_{d \le T} \sum_{ \bfrac{v (\mmod d)}{(v,d)=1}} \max_{y \le t} E(y;d,v)^2 \ll_B t T \log t.
\end{equation}

\noindent Hence, the contribution to $M_1(x)$ coming from $E(\sqrt{x-a^2};d,v)$ is given by
\begin{equation}
	\begin{split}
		&\sum_{d \le D} \chi_4(d) \sum_{\bfrac{u,v (\mmod d)}{\bfrac{(v,d)=1}{u^2+v^2 \equiv 0 (\mmod d)}}} \sum_{\bfrac{a \le \sqrt{x}}{a \equiv u(\mmod d)}} E(\sqrt{x-a^2};d,v)\\
			\ll & \sqrt{x}  \sum_{d \le D} \frac{1}{d} \sum_{\bfrac{v (\mmod d)}{(v,d)=1}} E^*(\sqrt{x};d,v) \sum_{\bfrac{u (\mmod d)}{u^2+v^2 \equiv 0(\mmod d)}}1 \\
			\ll & \sqrt{x} \sum_{d \le D} \frac{1}{d} \biggl(  \sum_{\bfrac{v (\mmod d)}{(v,d)=1}} E^*(\sqrt{x};d,v)^2\biggr)^{1/2} \biggl(  \sum_{\bfrac{v (\mmod d)}{(v,d)=1}} \biggl(\sum_{\bfrac{u (\mmod d)}{u^2+v^2 \equiv 0(\mmod d)}}1 \biggr)^2\biggr)^{1/2}\\
			\ll & \sqrt{x} \Biggr(\sum_{d \le D} \sum_{\bfrac{v (\mmod d)}{(v,d)=1}}E^*(\sqrt{x};d,v)^2 \Biggr)^{1/2} 
				\Biggr( \sum_{d \le D} \frac{1}{d^2} \sum_{\bfrac{v (\mmod d)}{(v,d)=1}} \Biggr( \sum_{\bfrac{u (\mmod d)}{u^2+v^2 \equiv 0 (\mmod d)}}  1\Biggr)^2\Biggr)^{1/2} \\
			\ll & \frac{x}{\log ^{\frac{A-5}{2}} x},
\end{split}
\end{equation}
where $E^*(\sqrt{x};d,v) = \sup_{a \le \sqrt{x}}  E(\sqrt{x-a^2};d,v)$. 
In the above we applied \eqref{BDH} for $A>1$
\begin{equation}
	\sum_{d \le D} \sum_{\bfrac{v (\mmod d)}{(v,d)=1}}E(\sqrt{x};d,v)^2 \ll_A \sqrt{x}D \log x = \frac{x}{\log^{A-1}x}
\end{equation}
and the trivial bound
\begin{equation}
	\begin{split}
		&\sum_{d \le D} \frac{1}{d^2} \sum_{\bfrac{v (\mmod d)}{(v,d)=1}} \Biggr( \sum_{\bfrac{u (\mmod d)}{u^2+v^2 \equiv 0 (\mmod d)}}  1\Biggr)^2 \\
		&\ll
		\sum_{d \le D} \frac{\varphi(d)}{d^2} \#\{x: \; x^2 \equiv -1 (\mmod d)\}^2 \ll \sum_{d \le D} \frac{2^{\om(d)}}{d} \ll \log^2 x.
	\end{split}
\end{equation}
 (the last can be seen, for example, by Levin-Fainleib theorem \cite{MR0229600}).

\noindent Thus, taking $A>5$ we obtain that for $B = (A-5)/{2}>0$ one has
\begin{equation} \label{Mest1}
	M_1(x) = 2 \sum_{d \leq D} \frac{\chi_4(d)}{\varphi(d)} \sum_{\bfrac{u,v (\mmod d)}{\bfrac{(v,d)=1}{u^2+v^2 \equiv 0 (\mmod d)}}}  \sum_{\bfrac{a \leq \sqrt{x}}{a \equiv u(\mmod d)}} \pi(\sqrt{x-a^2})+O\left(\frac{x}{\log^B x}\right).
\end{equation}

\noindent We continue working with the internal sum and get 
\begin{equation}
	\begin{split}
		  \sum_{\bfrac{a \leq \sqrt{x}}{a \equiv u(\mmod d)}} \sum_{p \le \sqrt{x-a^2}} 1 &= \sum_{p \le \sqrt{x}}  \sum_{\bfrac{a \leq \sqrt{x-p^2}}{a \equiv u(\mmod d)}} 1 = \sum_{p \le \sqrt{x}} \biggl(\frac{1}{d} \sum_{a \le \sqrt{x-p^2}}1 +O(1)\biggr) \\
		  &= \frac{1}{d} \biggl(\sum_{a^2+p^2 \le x}1+O\biggl( \frac{\sqrt{x}}{\log x} \biggr) \biggr).
	\end{split}
\end{equation}
The contribution of the error term coming from the big $O$ above is bounded via Lemma \ref{2^om/n_onA} and Lemma \ref{uv-sols} by
\begin{equation}
	\begin{split}
		\frac{\sqrt{x}}{\log x}  \sum_{d \le D} \frac{1}{\varphi(d)}  \sum_{\bfrac{u,v (\mmod d)}{\bfrac{(v,d)=1}{u^2+v^2 \equiv 0 (\mmod d)}}} 1 &= \frac{\sqrt{x}}{\log x}  \sum_{d \le D} 2^{\omega(d)} \fa(d) \\
		&\ll \frac{\sqrt{x}}{\log x} D \log D \ll \frac{x}{\log^A x}.
	\end{split}
\end{equation}

\noindent On applying \eqref{s1} and Lemma \ref{uv-sols} we get
\begin{equation}
		M_1(x) =  2 \sum_{d \leq D} \frac{\chi_4(d)}{d \varphi(d)} \sum_{\bfrac{u,v (\mmod d)}{\bfrac{(v,d)=1}{u^2+v^2 \equiv 0 (\mmod d)}}} \biggl( \frac{\pi}{2} \frac{x}{\log x} + O\biggl( \frac{x}{\log^2 x}\biggr) \biggr) + O \biggl( \frac{x}{\log^B x}\biggr).
\end{equation}

\noindent  By Lemma \ref{2^om/n_onA} and Lemma \ref{uv-sols} the error term coming from the sum is
\begin{equation}
	\ll \frac{x}{\log^2 x}  \sum_{d \leq D} \frac{1}{d \varphi(d)} \sum_{\bfrac{u,v (\mmod d)}{\bfrac{(v,d)=1}{u^2+v^2 \equiv 0 (\mmod d)}}} 1 \ll \frac{x}{\log^2 x}  \sum_{d \leq D} \frac{2^{\omega(d)}}{d} \fa(d) \ll \frac{x}{\log x}.
\end{equation}

\noindent Using  Lemma \ref{2^om/n_onA} and Lemma \ref{uv-sols} again  we conclude
\begin{equation}
	\begin{split}
		M_1(x) &=  \frac{\pi x}{\log x} \sum_{d \leq D} \frac{\chi_4(d)}{d \varphi(d)} \sum_{\bfrac{u,v (\mmod d)}{\bfrac{(v,d)=1}{u^2+v^2 \equiv 0 (\mmod d)}}} 1 +  O \biggl( \frac{x}{\log x}\biggr) \\
			&= \frac{\pi x}{\log x} \sum_{d \leq D}  \frac{2^{\omega(d)}}{d} \fa(d) + O \biggl( \frac{x}{\log x}\biggr) 
			= \frac{x}{2}+ O \biggl( \frac{x  \log \log x}{\log x}\biggr).
	\end{split}
\end{equation}

%% file: r0r2.tex
\section{Number of solutions to $p^2+q^2 = a^2+b^2 \le x$} \label{r_0r_2sec}
In this section we follow closely the technique developed in Section \ref{r_0r_1sec}.
As before let $A$ be some large positive integer and $D=\sqrt{x}/\log^A x$.
Then acting as in Section \ref{r_0r_1sec} we can write  $S_{0,2}(x) = 2M_2(x)+R_2(x)$, where
\begin{equation} \label{M2def}
	M_2(x) = \sum_{d \le D} \chi_4(d) \sum_{\bfrac{n \le x}{n \equiv 0(\mmod d)}} r_2(n), 
\end{equation}
and $R_2(x)$ corresponds to the sum over the middle divisors and can be bounded as in the previous section  
\begin{equation}
	R_2(x) \ll \frac{x}{\log^2 x} \log \log x.
\end{equation}

We split the quadratic congruence into two linear ones as
\begin{equation} \label{split}
		\sum_{\bfrac{n \le x}{n \equiv 0(\mmod d)}} r_2(n)  = \sum_{\bfrac{u,v (\mmod d)}{\bfrac{(uv,d)=1}{u^2+v^2 \equiv 0 (\mmod d)}}} \sum_{\bfrac{p^2+q^2 \le x}{\bfrac{ p \equiv u (\mmod d)}{q \equiv v(\mmod d)}}}1 + O(\omega(d) \sqrt{x})
\end{equation}
and bound the big $O$ contribution to $M_2(x)$ as
 \begin{equation}
 	\sum_{d \le D} \omega(d) \sqrt{x} \ll \sqrt{x} D \log D =\frac{x}{\log^{A-1}x}.
 \end{equation}
 
\noindent Further, for the internal sum one has
\begin{equation}
	\begin{split}
		\sigma_{u,v,d}(x) &=   	\sum_{\bfrac{p^2+q^2 \le x}{\bfrac{ p \equiv u (\mmod d)}{q \equiv v(\mmod d)}}}1 = \sum_{\bfrac{p \le \sqrt{x}}{p \equiv u (\mmod d)}} \sum_{\bfrac{q \le \sqrt{x-p^2}}{q \equiv v(\mmod d)}} 1 \\
			&=  \sum_{\bfrac{p \le \sqrt{x}}{p \equiv u (\mmod d)}} \biggl( \frac{1}{\varphi(d)} \sum_{q \le \sqrt{x-p^2}} 1 +E(\sqrt{x-p^2};d,v)\biggr).
	\end{split}
\end{equation}

\noindent By the Barban-Davenport-Halberstam theorem \eqref{BDH} and the Brun-Titchmarsh theorem applied to the sum over $p \le \sqrt{x}$ the contribution to $M_2(x)$ coming from $E(\sqrt{x-p^2};d,v)$ is bounded by
\begin{equation}
		   \frac{\sqrt{x}}{\log x}  \sum_{d \le D} \frac{1}{\varphi(d)} \sum_{\bfrac{v (\mmod d)}{(v,d)=1}} E(\sqrt{x};d,v) \sum_{\bfrac{u (\mmod d)}{u^2+v^2 \equiv 0(\mmod d)}}1 \ll  \frac{x}{\log ^{\frac{A-3}{2}} x}.
\end{equation}

\noindent Thus, we conclude that for any $B = (A-3)/2>0$ we have
 \begin{equation}
 	M_2(x) =  2 \sum_{d \le D} \frac{\chi_4(d)}{\varphi(d)} \sum_{\bfrac{u,v (\mmod d)}{\bfrac{(uv,d)=1}{u^2+v^2 \equiv 0 (\mmod d)}}}  \sum_{\bfrac{p \le \sqrt{x}}{p \equiv u (\mmod d)}} \sum_{q \le \sqrt{x-p^2}} 1 + O \biggl(  \frac{x}{\log^{B} x} \biggr).
 \end{equation}

\noindent Further, we write 
\begin{equation}
	\begin{split}
		 \sum_{\bfrac{p \le \sqrt{x}}{p \equiv u (\mmod d)}} \sum_{q \le \sqrt{x-p^2}} 1 & = \sum_{q \le \sqrt{x}}  \sum_{\bfrac{p \le \sqrt{x-q^2}}{p \equiv u (\mmod d)}} 1 \\
		 &= \frac{1}{\varphi(d)} \sum_{n \le x} r_2(n) + O \biggl( E(\sqrt{x};d,u) \frac{\sqrt{x}}{\log x}\biggr).
	\end{split}
\end{equation}

\noindent The contribution of $E(\sqrt{x};d,u)$ can be bounded exactly as the one above of $E(\sqrt{x};d,v)$, while for the main term we apply \eqref{s2} and get
\begin{equation}
	M_2(x) =   2 \sum_{d \le D} \frac{\chi_4(d)}{\varphi^2(d)} \sum_{\bfrac{u,v (\mmod d)}{\bfrac{(uv,d)=1}{u^2+v^2 \equiv 0 (\mmod d)}}} \biggl( \frac{ \pi x}{\log^2 x} + O\biggl(\frac{x \log \log x}{\log^3 x} \biggr)\biggr)+O \biggl(  \frac{x}{\log^{B} x} \biggr).
\end{equation}

\noindent  By Lemma \ref{2^om/phi(n)_onA} and Lemma \ref{uv-sols} the error term coming from the sum is
\begin{equation}
	\ll \sum_{d \leq D} \frac{1}{\varphi^2(d)} \sum_{\bfrac{u,v (\mmod d)}{\bfrac{(v,d)=1}{u^2+v^2 \equiv 0 (\mmod d)}}} \ll  \sum_{d \leq D} \frac{2^{\omega(d)}}{\varphi(d)} \fa(d) \ll  \log x.
\end{equation}

\noindent Using  Lemma \ref{2^om/n_onA} and Lemma \ref{uv-sols} again,  we conclude
\begin{equation}
	\begin{split}
		M_2(x) &=  \frac{ 2 \pi x}{\log^2 x} \sum_{d \leq D} \frac{\chi_4(d)}{\varphi^2(d)} \sum_{\bfrac{u,v (\mmod d)}{\bfrac{(v,d)=1}{u^2+v^2 \equiv 0 (\mmod d)}}} 1 +  O \biggl( \frac{x \log \log x}{\log^2 x}\biggr) \\
			&= \frac{\pi x}{\log^2 x} \sum_{d \leq D}  \frac{2^{\omega(d)}}{d} \fa(d) + O \biggl( \frac{x \log \log x}{\log^2 x}\biggr)  \\
			&= \frac{12 G}{\pi^2} \frac{x}{\log x} + O \biggl( \frac{x \log \log x}{\log^2 x}\biggr).
	\end{split}
\end{equation}

%% file: r1r2.tex
\section{Number of solutions to $a^2+p^2 = q^2+r^2 \le x$} \label{r_1r_2sec}

Notice that the lower bound follows immediately from $S_{1,2}(x) \ge S_{2,2}(x) \gg x/\log^2 x$, hence we focus on the upper bound here.

\noindent First, using \eqref{s2} and taking away the diagonal solutions one can write
\begin{equation} \label{n1def}
	\begin{split}
		S_{1,2}(x) &=  \sum_{n \le x}r_2^2(n)+N(x) =  \frac{2 \pi x}{\log^2 x}+N(x)+O\biggl(\frac{x \log \log x}{\log^3 x}\biggr),
	\end{split}
\end{equation}

\noindent where $N(x)$ counts the off-diagonal solutions to $a^2+p^2 = q^2+r^2 \le x$, namely, the ones with $\{a,p\} \neq \{q,r\}$.
We now wish to bound $N(x) \ll x/\log^2 x$. 
Notice that the contribution of $a=p$ or $q=r$ is trivially bounded by $x^{1/2}$, so we can assume $a \neq p$, $q \neq r$.
One can always assume that $q < r$ since otherwise we can just rename the variables.
When $q=2$ the number of quadruples in $N(x)$ is $\ll x^{1/2+\veps}$, thus we may neglect it and similarly for all the other variables.

Let $N_1(x)$ be the number of $(a,p,q,r) \in \N \times \P^3$ such that
\begin{equation} \label{main_r_1r_2}
	a^2+p^2=q^2+r^2 \le x, \quad  2 < a < q < r < p
\end{equation}

\noindent and $N_1'(x)$ be the number of $(a,p,q,r) \in \N \times \P^3$ such that
\begin{equation} \label{main2_r_1r_2}
	a^2+p^2=q^2+r^2 \le x, \quad  a, p \in (q,r), \, q > 2
\end{equation}
and $N_1''(x)$ be the number of $(a,p,q,r) \in \N \times \P^3$ such that
\begin{equation} 
	a^2+p^2=q^2+r^2 \le x, \quad  q, r \in (p,a), \, q > 2
\end{equation}

\noindent Then by the above we have $N(x) = N_1(x) + N_1'(x) + N_1''(x) + O(x^{1/2+\veps})$.

Let us work with $N_1(x)$. 
Since $q>2$, then $a$ and $q$ are necessarily of the same parity. 
Define $\P_{!2} = \P\setminus \{2\}$ to be the set of all odd prime numbers.
We change the variables to $\ell_1=(q-a)/2$, $\ell_2=(q+a)/2$ to obtain
\begin{equation} \label{n2def}
	N_1(x)=N_2(x)+O(x^{1/2+\veps}),
\end{equation}

\noindent where $N_2(x)$ is the number of  quadruples $(\ell_1,\ell_2,p,r)$ such that $2<r<p \le \sqrt{x}$ and
\begin{equation}\label{n2form}
	(\ell_2-\ell_1)^2+p^2 \le x, \quad \ell_1 < \ell_2, \quad \ell_1+\ell_2 \in \P_{!2}, \quad 4\ell_1\ell_2 = p^2-r^2. 
\end{equation}

\noindent Let $\Pc(x)$ be the set of $(p,r) \in \P^2$ such that $2<r<p \le \sqrt{x}$ and
\begin{equation}
	 r \overset{(1)} {\le} \frac{\sqrt{x}}{\log^{10}x} \quad\text{or}\quad p-r \overset{(2)} {<} \frac{\sqrt{x}}{\log^{10}x} \quad\text{or}\quad \sqrt{x}-\frac{\sqrt{x}}{\log^{10}x} \overset{(3)} {<} p.
\end{equation}

\noindent Then trivially $\#\Pc(x) \le \#\Pc_1(x)+\#\Pc_2(x)+\#\Pc_3(x)$, where $\Pc_i(x)$ is the set of $(p,r) \in \P^2$ with $2<r<p \le \sqrt{x}$ and the corresponding condition $(i)$.
Further, on applying Brun-Titchmarsh inequality we get
\begin{equation}
	\begin{split}
		\#\Pc_1(x) &\ll \pi(\sqrt{x}) \pi\biggl(\frac{\sqrt{x}}{\log^{10}x}\biggr) \ll \frac{x}{\log^{12}x},\\
		\#\Pc_2(x) &\ll \sum_{r \le \sqrt{x}} \sum_{\bfrac{p>r}{p-r<\sqrt{x}/{\log^{10}x}}} 1 \ll \frac{\sqrt{x}}{\log^{10}x} \sum_{r \le \sqrt{x}} \frac{1}{\log r} \ll  \frac{x}{\log^{12}x},\\
		\#\Pc_3(x) &\ll \pi(\sqrt{x}) \sum_{\bfrac{p \le \sqrt{x}}{p>\sqrt{x}-\sqrt{x}/\log^{10}x}}1 \ll \frac{x}{\log^{12}x}.
	\end{split}
\end{equation}

\noindent Hence, $\# \Pc (x) \ll x/\log^{12}x$. 
Using Cauchy-Schwarz inequality we deduce that the number of quadruples $(\ell_1,\ell_2,p,r)$ coming to $N_2(x)$ from $(p,r) \in \Pc(x)$ is bounded by

\begin{equation}
	\begin{split}
		\sum_{(p,r) \in \Pc(x)} \tau(p^2-r^2) &\ll \#\Pc^{1/2} \biggl(\sum_{r<p \le \sqrt{x}}\tau^2(p^2-r^2)\biggr)^{1/2} \\
		&\ll \frac{\sqrt{x}}{\log^6 x} \biggl(\sum_{d_1,d_2 \le 2 \sqrt{x}}\tau^2(d_1d_2)\biggr)^{1/2} \\
		&\ll \frac{\sqrt{x}}{\log^6 x} \sum_{d \le 2 \sqrt{x}}\tau^2(d) \ll \frac{x}{\log^3 x}.
	\end{split}
\end{equation}

\noindent Inserting it into \eqref{n2def} we obtain

\begin{equation} \label{N1N3}
	N_1(x) = N_3(x)+O \biggl( \frac{x}{\log^3 x}\biggr),
\end{equation}
where $N_3(x)$ counts the number of quadruples $(\ell_1,\ell_2,p,r)$ in $N_2(x)$ that satisfy $(p,r) \notin \Pc(x)$, i. e.

\begin{equation} \label{prcond}
	(p,r) \notin \Pc(x) \iff \frac{\sqrt{x}}{\log^{10}x} <r \le p-\frac{\sqrt{x}}{\log^{10}x} , \quad p \le \sqrt{x}-\frac{\sqrt{x}}{\log^{10}x}.
\end{equation}

\noindent Notice that by the symmetry of the equation we can exchange $r$ to $q$ in \eqref{main_r_1r_2} to get that the contribution of $(p,q) \in \Pc(x)$ is bounded by $x/\log^3 x$ as well. 
Hence, we have

\begin{equation}\label{N1N4}
	N_1(x) = N_4(x)+O \biggl( \frac{x}{\log^3 x}\biggr),
\end{equation}

\noindent where $N_4(x)$ counts the number of quadruples $(a,p,q,r)$ such that 
\begin{equation}
	a^2+p^2 = q^2+r^2 \le x, \quad 2 < a < q < r < p, \quad (p,r) \notin \Pc(x), \quad (p,q) \notin \Pc(x).
\end{equation}

Next, as in Rieger \cite{MR0229603} we transform our question into lattice counting problem.
Consider an equation 
\begin{equation}
	x_1^2+x_2^2=x_3^2+x_4^2 \le x, \quad x_i \in \mathbb{Z}_{>0}, \quad x_2 < x_1 < x_3.
\end{equation}

\noindent Notice that the condition $x_1 < x_3$ immediately implies $x_2>x_4$, hence our variables are ordered as $2 < x_4 < x_2 < x_1 < x_3$.

Further, on defining $2m_1=x_3-x_1 \ne 0$, $2m_2=x_2-x_4 \ne 0$ our equation becomes $2m_1(x_3+x_1)=2m_2(x_2+x_4)$.

\noindent Now let $d = (m_1,m_2)$ and $m_1=n_1d$, $m_2=n_2d$, $(n_1,n_2)=1$.
Since $n_1$ and $n_2$ are coprime it follows that $n_1$ divides $x_2+x_4$ and similarly $n_2$ divides $x_1+x_3$.
Denote $x_2+x_4=2n_1t$. Then we have $x_1+x_3=2n_2t$.
This implies that $x_i$ can be parametrized as 
\begin{eqnarray*}
	&x_1=n_2t-n_1d, &x_2=n_2d+n_1t,\\
	&x_3=n_1d+n_2t, &x_4=n_1t-n_2d.
\end{eqnarray*}
Thus, taking in the above 
\begin{equation}
	(x_1,x_2,x_3,x_4) = (r,q,p,a)
\end{equation}

\noindent we conclude that $N_4(x)$ is equal to the number of $(d,t,n_1,n_2) \in \mathbb{Z}_{>0}^4$ satisfying
\begin{equation}
	n_2t-n_1d, n_2d+n_1t, n_1d+n_2t \in \mathbb{P}; (n_1,n_2)=(d,t)=1; n_it \pm n_jd \le \sqrt{x}.
\end{equation}

\noindent Here in changing the variables we used essentially the fact that we have $a<q$ in \eqref{main_r_1r_2}, hence the corresponding form in $t,d$ has the minus sign.

Notice that by the above for any prime $s$ we have that $d, t \not \equiv 0 (\mmod s)$ simultaneously.

Let now $t, d$ be fixed with $(d,t)=1$, $u = \sqrt{x} \ge 2$ and $i \in \{1,2,3\}$. 
Denote for simplicity 
\begin{equation}
	\begin{split}
		&f_1=f_1(n_1,n_2)=n_2t-n_1d, \\
		&f_2=f_2(n_1,n_2)=n_2d+n_1t, \\
		&f_3=f_3(n_1,n_2)=n_1d+n_2t,\\
		&F(n_1,n_2) = \prod_{i=1}^{3}f_i(n_1,n_2).
	\end{split}
\end{equation}

\noindent Define 
\begin{equation}
	A^{t,d}(u) = \{ F(n_1,n_2) \colon f_i \le u, (n_1,n_2)=1 \}.
\end{equation}

\noindent Let $z \ge 2$ and define $P(z) = \prod_{p < z}p$.
Then
\begin{equation} \label{saz}
	S(A^{t,d}(u), z) = \sum_{\bfrac{n \in A^{t,d}(u)}{(n,P(z))=1}}1=\sum_{n \in A^{t,d}(u)} \sum_{\delta \mid (n,P(z))}\mu(\delta)=\sum_{\delta \mid P(z)}\mu(\delta) \sum_{\bfrac{n \in A^{t,d}(u) }{n \equiv 0 (\mmod \delta)}}1.
\end{equation}

\noindent Then the number of $(a,p,q,r) \in \N \times \P^3$ subject to \eqref{main_r_1r_2} is bounded by
\begin{equation} \label{n4sieve}
	N_4(x) \le \sum_{\bfrac{t,d \le u}{(t,d)=1}} S(A^{t,d}(u), z) + \frac{z}{\log z} u \max_{n \ll u^2}\tau(u),
\end{equation}
where the second summand corresponds to the number of solutions to $x_1^2+x_2^2 = x_3^2+x_4^2$, where $x_1 \le z$ is a prime and $x_2 \le u$ is an integer.

\noindent Denote the internal sum in \eqref{saz} by $A^{t,d}_{\delta}(u)$. 
We have 
\begin{equation}
	A^{t,d}_{\delta}(u) = \# \{ n_1,n_2  \colon F(n_1,n_2) \equiv 0 (\mmod \delta), f_i \le u\}.
\end{equation}

\noindent Define $\nu(\delta) = \# \{ n_1, n_2 (\mmod \delta) \colon  F(n_1,n_2) \equiv 0 (\mmod \delta)\}$.
By the Chinese remainder theorem the function $\nu(\delta)$ is multiplicative.
Since $f_i \le u$ and $f_2 < f_1 < f_3 = n_1d+n_2t \le u$, one has $n_1 \le u/d$, $n_2 \le u/t$.

\noindent We thus have
\begin{equation} \label{adeltamain}
	\begin{split}
		A^{t,d}_{\delta}(u) &= \frac{\nu(\delta)}{\delta^2} \frac{u^2}{td} + O \biggl(  \frac{\nu(\delta)}{\delta^2} \frac{u}{\min(t,d)} \biggr).
	\end{split}
\end{equation}

Let us compute $\nu(p)$ for different values of $p$. 
Notice that $p$ can not simultaneously divide $t,d$ since $(t,d)=1$. 

Consider first the case $td \equiv 0 (\mmod p)$, $p \neq 2$. 
Then $f_i \equiv 0 (\mmod p)$ has $p$ solutions.
Suppose that $t \equiv 0 (\mmod p)$, $d \not \equiv 0 (\mmod p)$, the opposite case can be treated analogously.
Then $f_1, f_2 \equiv 0 (\mmod p)$, $f_2, f_3 \equiv 0 (\mmod p)$ and $f_1,f_2,f_3 \equiv 0 (\mmod p)$ have each one solution $n_1 \equiv n_2 \equiv 0(\mmod p)$. 
Finally $f_1,f_3 \equiv 0(\mmod p)$ is equivalent to $\pm n_1 \equiv 0(\mmod p)$ and has $p$ solutions.
Thus $\nu(p) = 3p-2-p+1 = 2p - 1$.

Next, let $p \neq 2$ and $t,d \not \equiv 0 (\mmod p)$.
Then the congruence $f_i \equiv 0 (\mmod p)$ has $p$ solutions.
Next $f_1,f_3 \equiv 0 (\mmod p)$ and $f_1,f_2,f_3 \equiv 0(\mmod p)$ both have one solution, namely $n_1 \equiv n_2 \equiv 0(\mmod p)$.
The remaining cases $f_1, f_2 \equiv 0 (\mmod p)$ and $f_2,f_3 \equiv 0 (\mmod p)$ are slightly more complicated.
Consider $f_1, f_2 \equiv  0 (\mmod p)$, the other one is analogous.
From $f_1 \equiv 0 (\mmod p)$ we obtain $n_2 \equiv n_1 d t^{-1} (\mmod p)$ and inserting it into $f_2 \equiv 0 (\mmod p)$ one concludes $n_1 (d^2 t^{-1} +t) \equiv 0 (\mmod p)$.
Further, if $d^2 t^{-1} + t \equiv 0 (\mmod p)$, then we have $p$ solutions and hence $\nu(p) = 2p-1$, while if $d^2 t^{-1} +t  \not \equiv 0 (\mmod p)$, then we have only one solution, namely, $n_1 \equiv n_2 \equiv 0 (\mmod p)$ and hence $\nu(p) = 3p - 2$.

We now move to $p=2$. 
Let $t,d \not \equiv 0 (\mmod 2)$.
Then $t,d \equiv 1 (\mmod 2)$ and $f_i \equiv 0 (\mmod 2)$ has each $2$ solutions, namely, $n_1 \equiv \pm n_2 (\mmod 2)$.
Same holds for $f_i, f_j \equiv 0 (\mmod 2)$ and $f_1,f_2,f_3 \equiv 0 (\mmod 2)$.
Thus $\nu(p) = 2$ in that case. 
Consider now $td \equiv 0 (\mmod 2)$.
Assume $t \equiv 0 (\mmod 2)$ and $d \equiv 1 (\mmod 2)$, the other case is analogous.
Then $f_i \equiv 0 (\mmod 2)$ has each $2$ solutions, namely, $n_i \equiv 0(\mmod 2)$, $n_j$ any; 
$f_1, f_2 \equiv 0 (\mmod 2)$, $f_2, f_3 \equiv 0 (\mmod 2)$ and $f_1, f_2, f_3 \equiv 0 (\mmod 2)$ have one solution $n_1 \equiv n_2 \equiv 0 (\mmod 2)$, 
while $f_1,f_3 \equiv 0(\mmod 2)$ has two solutions, namely, $n_1 \equiv 0 (\mmod d)$, $n_2$ any.
Hence in that case we get $\nu(p) = 3$.

Finally, we conclude
\begin{equation}
	\nu(p) = \begin{cases}
						3p-2, & p \not = 2; t,d \not \equiv 0 (\mmod p), d \not \equiv \pm it (\mmod p), d \not \equiv \pm t (\mmod p),\\
						2p-1, & p \not = 2; t,d \not \equiv 0 (\mmod p), d \equiv \pm it (\mmod p), d \not \equiv \pm t (\mmod p),\\
						2p-1, & p \not = 2; t,d \not \equiv 0 (\mmod p), d \not \equiv \pm it (\mmod p), d \equiv \pm t (\mmod p),\\
						2p-1, & p \not = 2; td \equiv 0 (\mmod p),\\
						3, & p = 2; td \equiv 0 (\mmod p),\\
						2, & p = 2; t,d \not \equiv 0 (\mmod p),
					\end{cases}
\end{equation}

\noindent where $i$ is defined as $i^2 \equiv -1(\mmod p)$.
Summing \eqref{adeltamain} over $t,d$ we get
\begin{equation} \label{sumovertd}
	\begin{split}
		\sum_{\bfrac{t,d \le u}{(t,d)=1}}A^{t,d}_{\delta}(u) &= u^2  \sum_{\bfrac{t,d \le u}{(t,d)=1}} \frac{1}{td} \sum_{\delta \mid P(z)} \frac{\mu(\delta) \nu(\delta)}{\delta^2}+ O \biggl( \sum_{\bfrac{t,d \le u}{(t,d)=1}} \sum_{\delta \mid P(z)} \frac{\nu(\delta)}{\delta^2} \frac{u}{\min(t,d)} \biggr).
	\end{split}
\end{equation}

\noindent In the main term we have

\begin{equation}
	 	\sum_{\delta \mid P(z)} \frac{\mu(\delta) \nu(\delta)}{\delta^2} = \prod_{p < z} \biggl(1-\frac{\nu(p)}{p^2} \biggr) = G(t,d) V(z),
\end{equation}

\noindent where
\begin{equation}
	\begin{split}
		V(z) &= \prod_{2 < p < z} \biggl(1-\frac{3p-2}{p^2} \biggr),\\ 
		G(t,d) &= \biggl(1-\frac{\nu(2)}{4} \biggr) \prod_{\substack {2 < p <z \\ d \equiv \pm i t (\mmod p) \\ \text{or} \\ d \equiv \pm t (\mmod p) }} \biggl(1+\frac{p-1}{p^2-3p+2} \biggr).
	\end{split}
\end{equation}

\noindent By Mertens' theorem we have
\begin{equation} \label{mertensvz}
	V(z) = \prod_{2 < p < z} \biggl(1 - \frac{3}{p} + O \biggl(\frac{1}{p^2}\biggr)\biggr) = \frac{C+o(1)}{\log^3 z}.
\end{equation}

\noindent Thus \eqref{sumovertd} becomes
\begin{equation}
	\sum_{\bfrac{t,d \le u}{(t,d)=1}}A^{t,d}_{\delta}(u) = u^2 V(z) \sum_{\bfrac{t,d \le u}{(t,d)=1}} \frac{G(t,d)}{td} + O \biggl( \sum_{\bfrac{t,d \le u}{(t,d)=1}} \sum_{\delta \mid P(z)} \frac{\nu(\delta)}{\delta} \frac{u}{\min(t,d)} \biggr).
\end{equation}

We now use that $(p,r) \notin \Pc(x)$, see \eqref{prcond}.
Recall that $x_3-x_1 = 2n_1d = p-r \ge \sqrt{x} / \log^{10}x$ by \eqref{prcond}.
Also $x_2+x_4 = 2n_1t = a+q \le 2\sqrt{x}$.
Thus $2d/t \ge 1/\log^{10}x$.
Similarly $x_1+x_3 = 2n_2 t = p+r \ge 3\sqrt{x}/\log^{10}x$, $x_2-x_4 = 2n_2d < \sqrt{x}$ imply $t/d \ge 3/\log^{10}x$.

Consider first $d \not \equiv \pm t (\mmod p)$,  $d \not \equiv \pm i t (\mmod p)$ for any $p < z$.
Then $G(t,d) = 1$ and we have for $d \le t$
\begin{equation}
	\sum_{\bfrac{t,d \le u}{(t,d)=1}} \frac{G(t,d)}{td} = \sum_{t \le u} \frac{1}{t} \sum_{\bfrac{t/(2\log^{10}x) \le d \le t}{(d,t)=1}} \frac{1}{d} \ll \log u (\log \log x),
\end{equation}
while for $d \ge t$ we similarly have
\begin{equation}
	\sum_{\bfrac{t,d \le u}{(t,d)=1}} \frac{G(t,d)}{td} = \sum_{d \le u} \frac{1}{d} \sum_{\bfrac{3d/\log^{10}x \le t \le d}{(d,t)=1}} \frac{1}{t} \ll \log u (\log \log x).
\end{equation}

\noindent Further, on defining
\begin{equation}
	\Psi(n) = n \prod_{p \mid n} \biggl(1+\frac{1}{p}\biggr)
\end{equation}
we have that $G(t,d) \ll \Psi(z)/z$ (in fact, this bound holds for all $t,d$) and thus in the remaining cases
\begin{equation} \label{Gtdbound}
	\sum_{\bfrac{t,d \le u}{(t,d)=1}} \frac{G(t,d)}{td} \ll \frac{\Psi(z)}{z} \log u (\log \log x) \ll \log u (\log \log x) (\log \log z),
\end{equation}
where we used that $\Psi(n)\varphi(n) < n^2$ and hence $\Psi(n)/n < n /\varphi(n) \ll \log \log n$.

\noindent Acting similarly in the error term of \eqref{sumovertd} we obtain that the term in the big-$O$ satisfies
\begin{equation}
	\begin{split}
		\sum_{\bfrac{t,d \le u}{(t,d)=1}} \sum_{\delta \mid P(z)} \frac{\nu(\delta)}{\delta} \frac{u}{\min(t,d)} &= u V(z) \sum_{\bfrac{t,d \le u}{(t,d)=1}} \frac{G(t,d)}{\min(t,d)} \le u^2 V(z) \sum_{\bfrac{t,d \le u}{(t,d)=1}} \frac{G(t,d)}{td} \\
		&\ll \frac{u^2}{\log^3 z} \log u (\log \log x)(\log \log z),
	\end{split}
\end{equation} 
where we applied \eqref{Gtdbound} and \eqref{mertensvz}.
Taking $z=x^{\alpha}$ for $\alpha$ small enough we conclude
\begin{equation}
	\sum_{\bfrac{t,d \le u}{(t,d)=1}}A^{t,d}_{\delta}(u) \ll \frac{x}{\log^2 x} (\log \log x)^2,
\end{equation}
where we bounded $(z/\log z) u \max_{n \ll u^2} \tau(n) \ll x^{1/2+\veps}/\log x \ll x/\log^2 x$.
\noindent What remains is to estimate $N_1'(x)$, $N_1''(x)$, see \eqref{main2_r_1r_2}.
This goes completely analogous, we change variables to $\ell_1 = (a-q)/2$, $\ell_2=(a+q)/2$, then bound away the contribution of $(r,p) \in \Pc(x)$ and estimate $N_4'(x)$ given as
\begin{equation}
	N_4'(x) = \#\{ a^2+p^2 = q^2+r^2 \le x, \quad q > 2, \, a,p \in (q,r), \quad (r,p) \notin \Pc(x) \}
\end{equation}
and similarly for $N_4''(x)$ with the contribution of $(q,p) \in \Pc(x)$ taken away.
Finally we conclude
\begin{equation}
	N_4(x) \ll  \frac{x}{\log^2 x} (\log \log x)^2.
\end{equation}